\documentclass[11pt]{article}
\usepackage[mathvars]{brsheader} 
\abbreviateBBMathFont

\newcommand{\xto}[1]{\xrightarrow{#1}}
\newcommand{\MRto}{\xrightarrow[\text{MR}]{}}

\title{The regular pentagon is canonically Ramsey}
\author{Benedict \textsc{Randall Shaw}\footnotemark[1]}
\date{August 2026}

\begin{document}

\maketitle

\renewcommand{\thefootnote}{\fnsymbol{footnote}}
\footnotetext[1]{\href{mailto:bwr26@cam.ac.uk}{bwr26@cam.ac.uk}, Department of Pure Mathematics and Mathematical Statistics (DPMMS), University of Cambridge, Wilberforce Road, Cambridge, CB3 0WA, United Kingdom}

\begin{abstract}
A set of points \(C\subset \mathbb{R}^n\) is \textit{canonically Ramsey} if there is some larger set of points \(S\subset \mathbb{R}^{n'}\) such that any colouring of \(S\) contains either a monochromatic copy of \(C\) or a rainbow copy of \(C\). Mao, Ozeki, and Wang \cite{mao2022euclideangallairamseytheory} introduced this notion, showing that the 30-60-90 triangle is canonically Ramsey. Since then, many other configurations have been shown to be canonically Ramsey. The author \cite{randallshaw2026cuboidscanonicallyramsey} showed that cuboids are canonically Ramsey. Ge, Shu, Xu, and Yu \cite{ge2026simplicesexhibitcanonicalramsey} later showed that all simplices are canonically Ramsey, after which the author \cite{randallshaw2026productssimplicescanonicallyramsey} showed that all products of simplices are canonically Ramsey, a class which, together with its closure under taking subsets, includes all previously known canonically Ramsey sets. We prove that regular polygons with a prime number of sides are canonically Ramsey---the first known sets outside this class.
\end{abstract}

\section{Introduction}

We consider \textit{configurations}, finite sets of points in some space \(\mathbb{R}^n\). We are not concerned with the dimension of the ambient space, and so for \(n<n'\), we identify \(\mathbb{R}^n\) with a subset of \(\mathbb{R}^{n'}\) in the usual way. For two configurations \(C,C'\), we say some \(X\subset C'\) is a \textit{copy} of \(C\) if there is a surjective isometry mapping \(C\) to \(X\), or a \textit{scaled copy} of \(C\) if there is a surjective isometry mapping some dilation of \(C\) to \(X\). We study \textit{colourings} of configurations, maps from the set of points in a configuration to some other set on which we assume no further structure---equivalently, we may think of these as partitions of the point set, or in turn as equivalence relations on the point set.

Given \(S\subset \mathbb{R}^n\) for some \(n\), a configuration \(C\), and some positive integer \(r\), we write
\[S\xto{r} C\]
if any colouring of \(S\) with at most \(r\) colours contains a \textit{monochromatic} copy of \(C\)---that is, a copy of \(C\) in \(S\) all of whose points receive the same colour. Equivalently, in the language of partitions, a monochromatic copy is a copy contained in a single part of the partition---or, in the language of equivalence relations, a copy all of whose points are equivalent. A configuration \(C\) is called \textit{Ramsey} if, for any \(r\), there is some configuration \(S\) such that \(S\xto{r}C\). By a compactness argument, this is equivalent to asking whether there is some \(n\) such that \(\mathbb{R}^n\xto{r} C\). The question of which sets are Ramsey is central to Euclidean Ramsey theory, and was first studied by Erdős, Graham, Montgomery, Rothschild, Spencer, and Straus \cite{erdos1973}, who showed that all Ramsey sets are spherical, and that any product of Ramsey sets is itself Ramsey. Frankl and Rödl \cite{frankl1990} later proved that (nondegenerate) simplices are Ramsey.

Kříž \cite{kriz1991} then proved that \textit{soluble} sets---sets acted on transitively by some soluble group of symmetries---are Ramsey, and thus that all \textit{subsoluble} sets---sets contained in some soluble set---are also Ramsey. This class of sets includes, among many others, every regular polygon. Later, Karamanlis \cite{karamanlis2022} showed that this class also includes all simplices, and indeed Behague \cite{behague2025nearlyknowneuclideanramsey} has shown that nearly all known Ramsey sets are subsoluble. The question of exactly which sets are Ramsey remains open.

In canonical Euclidean Ramsey theory, we remove the bound on the number of colours we use. Of course, we can no longer guarantee a monochromatic copy of any non-trivial configuration, since it would be possible just to colour every point a different colour. We are thus motivated towards the following definition: given \(S\subset \mathbb{R}^n\) for some \(n\) and a configuration \(C\), we write
\[S\MRto C\]
if any colouring of \(S\) contains either a monochromatic copy of \(C\) or a \textit{rainbow} copy of \(C\)---that is, a copy of \(C\) all of whose points receive different colours. Equivalently, in the language of partitions, a rainbow copy is a copy all of whose points lie in different parts of the partition---or, in the language of equivalence relations, a copy none of whose points are equivalent to any other. A configuration \(C\) is called \textit{canonically Ramsey} if there is some configuration \(S\) such that \(S\MRto C\). By a compactness argument, this is again equivalent to asking whether there is some \(n\) such that \(\mathbb{R}^n\MRto C\). Note that both the class of Ramsey sets and the class of canonically Ramsey sets are closed under taking subsets.

The study of canonical Euclidean Ramsey theory is relatively new, but has been quite active. The question of which sets are canonically Ramsey was first studied by Mao, Ozeki, and Wang \cite{mao2022euclideangallairamseytheory}, who showed that the 30-60-90 triangle is canonically Ramsey. Cheng and Xu \cite{cheng2023euclideangallairamseyvariousconfigurations} then proved that squares, right-angled triangles, and a wide family of other triangles and simplices are canonically Ramsey. Gehér, Sagdeev, and Tóth \cite{geher2026} subsequently proved that all hypercubes are canonically Ramsey. Fang, Ge, Shu, Xu, Xu, and Yang \cite{fang2025canonicalramseytrianglesrectangles} then showed that all triangles, all rectangles, and a wider class of tetrahedra than previously known are canonically Ramsey. The author \cite{randallshaw2026cuboidscanonicallyramsey} then proved that all cuboids are canonically Ramsey, implying that a wider class of simplices are canonically Ramsey. Ge, Shu, Xu, and Yu \cite{ge2026simplicesexhibitcanonicalramsey} then showed that all simplices are canonically Ramsey, after which the author \cite{randallshaw2026productssimplicescanonicallyramsey} showed that all products of simplices are canonically Ramsey---a class which included all known canonically Ramsey sets.

Fang, Ge, Shu, Xu, Xu, and Yang \cite{fang2025canonicalramseytrianglesrectangles} have conjectured that all Ramsey sets should be canonically Ramsey, but both this conjecture and its converse remain open. By comparison with the progress of classical Euclidean Ramsey theory, one would hope to work towards an analogue of Kříž's theorem that subsoluble sets are Ramsey. The smallest natural next set to study is therefore the regular pentagon. We show that this is indeed canonically Ramsey. Rather surprisingly, however, this proof does not generalise to all regular polygons, but only those with a prime number of sides. Thus our main result is the following:

\begin{theorem}\label{thm:main}
    For any prime \(p\), the regular \(p\)-gon is canonically Ramsey.
\end{theorem}

In fact, we prove that all powers of a regular \(p\)-gon are canonically Ramsey. It is not yet known in general whether products of canonically Ramsey sets are themselves Ramsey---if so, this would imply that all canonically Ramsey sets are contained in arbitrarily large canonically Ramsey sets, and thus are themselves Ramsey. This would have consequences of significant interest, as this would mean that \(\{0,1,2\}\), for example, is not canonically Ramsey---but one can show that any colouring witnessing this cannot be spherical, and cannot even be preserved by permuting coordinates. Hence this is genuinely a stronger result than just showing a regular polygon with a prime number of sides is canonically Ramsey:

\begin{theorem}\label{thm:mainplus}
    For any prime \(p\) and positive integer \(k\), where \(C\) is a regular \(p\)-gon, there is some \(n\) such that
    \[p^{-p/2} C^n \MRto C^k.\]
\end{theorem}

Note that for \(n\leq 4\), regular \(n\)-gons and powers thereof are already known to be canonically Ramsey \cite{geher2026,randallshaw2026productssimplicescanonicallyramsey}, and so we assume \(p\geq 5\).

\section{Notation and outline}

Let \(C=\{x_1,\dots,x_p\}\) be the points of a regular \(p\)-gon with unit side length, where the points are numbered in order around the polygon. To make the notation easier, we will often identify \(C\) with \([p]\), so that \(x_i\) is identified with \(i\). Thus we will similarly identify \(C^n\) with \([p]^n\), and write \((a_1,\dots,a_n)\) to mean \(\left(x_{a_1},\dots,x_{a_n}\right)\). We will often in fact be working with scaled copies of \(C^n\)---where there is no confusion, we may omit this scaling from our notation. One particular scaled copy we will use at several points is the embedding \(\sigma_n\) of \(\sqrt{p} C^n\) into \(C^{pn}\) that maps \((a_1,\dots,a_n)\) to
\[\left(a_1,a_1+1,\dots,a_1+(p-1),a_2,\dots,a_n,a_n+1,\dots,a_n+(p-1)\right),\]
where entries are considered modulo \(p\). We call this the \textit{standard emulated copy of} \(\sqrt{p}C^n\), and omit the subscript on \(\sigma\) where there is no confusion. For \(n'>pn\), we consider instead the map \(\sigma_{n,n'}\) that maps \((a_1,\dots,a_n)\) to
\[\left(a_1,a_1+1,\dots,a_1+(p-1),a_2,\dots,a_n,a_n+1,\dots,a_n+(p-1),1,\dots,1\right),\]
and call this the \textit{standard emulated copy of} \(\sqrt{p}C^n\) \textit{within} \(C^{n'}\).

For our purposes, we think of a colouring of \(C^n\) not as a function from \(C^n\) to any particular set, but instead as an equivalence relation \(\sim\) on \(C^n\), where two points are related if they have the same colour. Recalling that we have identified \(C^n\) with \([p]^n\), we will want to consider `local' fragments of this equivalence relation in the following way: for any word \(w\in\{1,2,\dots,p,*\}^n\) containing \(m\) instances of the special character \(*\), we consider a corresponding injection \(\iota_w: [p]^m\to [p]^n\) that maps \((a_1,\dots,a_m)\) to the word obtained by replacing the \(m\) instances of \(*\) in \(w\) by \(a_1,\dots,a_m\) in that order. This induces an equivalence relation on \([p]^m\), the pullback of \(\sim\) with respect to \(\iota_w\)---we call this \(\sim_w\). So \(\sim_w\) is the equivalence relation on \([p]^m\) with \(a\sim_w b\) if \(\iota_w(a)\sim \iota_w(b)\). As a shorthand, we say this word \(w\) and equivalence relation \(\sim_w\) each have \textit{dimension} \(m\).

Note that we may also extend \(\iota_w\) to a map from \(\{1,2,\dots,p,*\}^m\) to \(\{1,2,\dots,p,*\}^n\), by substituting entries similarly---we may use this extension, but it will not affect the definition of \(\sim_w\), as the equivalence relation \(\sim\) is not defined on words containing \(*\). Notice that in this convention, where the length of \(w'\) is the dimension of \(w\), we may compose \(\iota_w\) and \(\iota_{w'}\), and this composition \(\iota_w\circ\iota_{w'}\) is exactly the inclusion \(\iota_{w''}\) given by \(w''=\iota_w(w')\). In particular, we then have \(\sim_{w''}=(\sim_w)_{w'}\).

Note that we may similarly extend \(\sigma_n\) to a map from \(\{1,\dots,p,*\}^n\) to \(\{1,\dots,p,*\}^{pn}\) that takes words of dimension \(m\) to words of dimension \(pm\) as follows: when \(w_i\in [p]\), we choose the entries at positions \((i-1)p+1,\dots,ip\) of \(\sigma(w)\) to be \(w_i,w_i+1,\dots,w_i+(p-1)\) as usual. Otherwise, if \(w_i=*\), then we simply set all these entries of \(\sigma(w)\) to be \(*\). So if \(\sim'\) is the pullback of \(\sim\) on \([p]^n\) with respect to \(\sigma\), then certainly \(\sim'_w\) depends only on \(\sim_{\sigma(w)}\), as the image of \(\sigma\circ \iota_w\) is entirely contained in the image of \(\iota_{\sigma(w)}\).

The proof divides into two main steps. First, we show that for sufficiently large \(n'\), given any equivalence relation \(\sim\) on \(C^{n'}\), we may select a scaled copy of \(C^n\) in \(C^{n'}\) within which each induced \(\sim_w\) depends only on the dimension of \(w\)---so for each dimension \(m\), there is some \(\sim_m\) such that \(\sim_w=\sim_m\) for each word \(w\in\{1,\dots,p,*\}^n\) of dimension \(m\). We say that such a \(C^n\) is \(\sim\)-\textit{invariant}. We prove this by an induction analogous to that of Kříž \cite{kriz1991}.

We then work within this copy of \(C^n\). Since \(\sim_w\) depends only on its dimension, we suppress its subscript where the meaning is clear. We show that whenever two points are equivalent, and differ in some coordinate, with values \(a\) and \(b\) at that coordinate, then we have \(ab\sim ba\). We show that this identity is transitive: if \(ab\sim ba\) and \(bc\sim cb\), then \(ac\sim ca\). But now consider the standard emulated copy of \(\sqrt{p}C^k\), the set of all points of the form
\[\left(a_1,a_1+1,\dots,a_1+(p-1),a_2,\dots,a_k,a_k+1,\dots,a_k+(p-1),1,\dots,1\right),\]
where the entries are considered modulo \(p\). Suppose that this copy is not rainbow, so that two of these points are equivalent. But then there are some \(a\neq b\) such that for each \(i\),
\[(a+i)(b+i)\sim (b+i)(a+i).\]
By transitivity, this is enough to show that \(ab\sim ba\) for all choices of \(a\) and \(b\). But then any permutation of a point's coordinates does not change its colour---so in particular, the scaled copy of \(C^k\) described above is monochromatic. This gives the desired result.

\section{Finding a \(\sim\)-invariant \(C^n\)}

We first define two slightly weaker notions of invariance. First, for any \(S\subset [p]\), we say that an equivalence relation \(\sim\) on \(C^n\) is \(S\)-\textit{interchangeable} if, for any \(w\in\{1,\dots,p,*\}^n\), whenever an entry of \(w\) that is in \(S\) is replaced by another member of \(S\), the induced relation \(\sim_w\) is unchanged. So for example, when \(S\) is a singleton, every equivalence relation is \(S\)-interchangeable.

Secondly, for any \(S\subset [p]\), we say that an equivalence relation \(\sim\) on \(C^n\) is \(S\)-\textit{swappable} if, for any \(w\in\{1,\dots,p,*\}^n\), whenever two adjacent entries of \(w\) are swapped, and at least one of those entries is in \(S\), the induced relation \(\sim_w\) is unchanged. So for example, if an equivalence relation is \([p]\)-interchangeable, then any adjacent transposition leaves \(\sim_w\) unchanged, as swapping \(*\) and \(*\) has no effect, and so by extension any permutation does not change \(\sim_w\). But this would not necessarily be true for \([p-1]\)-interchangeable relations, as it would not cover swapping \(p\) and \(*\). 

Observe that if an equivalence relation \(\sim\) is \(S\)-interchangeable, and for some \(x\in S\), it is \(\{x\}\)-swappable, then certainly it is also \(S\)-swappable, as any swap including a member of \(S\) can be realised by changing that member to an \(x\), making the swap, and then changing it back. Also note that an equivalence relation that is both \([p]\)-interchangeable and \([p]\)-swappable must in fact be \(\sim\)-invariant, as any word \(w\) can be made into any other word of the same dimension by permuting the entries so the \(*\)s are in the right coordinates, and then changing the other entries.

We first prove the steps we will require for our induction as two separate lemmas:

\begin{lemma}\label{lem:canswap}
    For any \(n\), and any \(r \in [p]\), there exists \(n'\) such that, for any \([r]\)-interchangeable equivalence relation \(\sim\) on \(C^{n'}\), there is some copy of \(C^n\) on which \(\sim\) is both \([r]\)-interchangeable and \([r]\)-swappable.
\end{lemma}
\begin{proof}
    Let \(t\) be the number of possible equivalence relations on \([p]^{n-1}\), and note that \(t\) is finite. By Ramsey's theorem \cite{ramsey1930}, we may choose \(n'\) such that any \(t\)-colouring of the subsets of \([n']\) of size \(n-1\) contains a monochromatic set of size \(n\)---that is to say, some set of size \(n\) all of whose subsets of size \(n-1\) receive the same colour.

    For each \(S\subset [n']\), we consider the word \(w^S\) which has \(w^S_i=*\) for each \(i\in S\), and \(w^S_i=1\) otherwise. We colour the subsets \(S\) of size \(n-1\) by the equivalence relation \(\sim_{w^S}\). Now this is a \(t\)-colouring, so certainly there is some \(X\subset [n']\) of size \(n\) such that all subsets of \(X\) of size \(n-1\) receive the same colour. Now the image \(\iota_{w^X}\) gives our desired copy of \(C^n\). Clearly the restriction of \(\sim\) to this copy of \(C^n\) is still \([r]\)-interchangeable---it just remains to show it is \([r]\)-swappable.

    Indeed, by our earlier observation, it suffices to show that \(\sim\) is \(\{1\}\)-swappable on this copy of \(C^n\). Suppose that \(w,w'\in\{1,\dots,p,*\}^n\) differ only by a swap of characters in positions \(i\) and \(i+1\), and w.l.o.g.\ \(w_i=w'_{i+1}=1\). But now consider \(w''\) of length \(n-1\) obtained from \(w\) by deleting the \(i\)th character, or equivalently, from \(w'\) by deleting the \((i+1)\)th character. Notice that for \(S=[n]\setminus\{i\}\) and \(S'=[n]\setminus\{i+1\}\), we have \(w=\iota_{w^S}(w'')\), and \(w'=\iota_{w^{S'}}(w'')\). So in particular,
    we find that \(\iota_w=\iota_{w^S}\circ \iota_{w''}\), and correspondingly \(\iota_{w'}=\iota_{w^{S'}}\circ \iota_{w''}\).
    
    Now \(\sim_w\) is exactly \((\sim_{w^S})_{w''}\), and correspondingly \(\sim_{w'}\) is \((\sim_{w^{S'}})_{w''}\). But by the choice of our copy of \(C^n\), we know that \(\sim_{w^S}=\sim_{w^{S'}}\). Hence we find that \(\sim_w=\sim_{w'}\), as desired. Thus \(\sim\) is \(\{1\}\)-swappable on this copy of \(C^n\), and so must also be \([r]\)-swappable, as desired.
\end{proof}

\begin{lemma}\label{lem:canchange}
    For any \(n\), any \(r \in \{2,\dots,p\}\), and any equivalence relation \(\sim\) on \(C^{pn}\) that is \([r-1]\)-swappable, \(\sim\) is \([r]\)-interchangeable on the standard emulated copy of \(\sqrt{p}C^n\).
\end{lemma}
\begin{proof}
    Let \(w,w'\in\{1,\dots,p,*\}^n\) differ only in position \(i\), such that both \(w_i,w'_i\in [r]\). Let \(\sim'\) be the equivalence relation \(\sim\) induces on the standard emulated copy of \(\sqrt{p}C^n\). It will suffice to show that \(\sim'_w=\sim'_{w'}\).

    Since these depend only on \(\sim_{\sigma(w)}\) and \(\sim_{\sigma(w')}\) respectively, it will be enough to show that these two equivalence relations are the same. Indeed, because \(\sim\) is \([r-1]\)-swappable, we need only show that \(\sigma(w)\) can be made into \(\sigma(w')\) by a sequence of the relevant changes and swaps permitted by these conditions.

    \(\sigma(w)\) and \(\sigma(w')\) differ only in positions \((i-1)p+1,\dots,ip\), where \(\sigma(w)\) has entries \(w_i,w_i+1,\dots,p,1,\dots,w_i-1\), and \(\sigma(w')\) has corresponding entries \(w'_i,w'_i+1,\dots,p,1,\dots,w'_i-1\). Suppose without loss of generality that \(w'_i<w_i\). But now \(w'_i\leq r\), so certainly \(w_i,\dots,w'_i-1\) are all in \([r-1]\). Hence we can swap any of these with any other character in \([p]\) without changing the corresponding equivalence relation (initially \(\sim_{\sigma(w)}\)). A sequence of adjacent transpositions lets us move \(w'_i-1\) to the front of this sequence of entries, and then \(w'_i-2\), and so on, until we have moved \(w_i\) to the front of this sequence. But then we will have formed \(\sigma(w')\)---and so by our assumption of \([r-1]\)-swappability, we find that \(\sim'\) is \([r]\)-interchangeable, as desired.
\end{proof}

We now construct a \(\sim\)-invariant \(C^n\).
\begin{lemma}\label{lem:invariant}
    For any \(n\), there is some \(n'\) such that given any equivalence relation \(\sim\) on \(C^{n'}\), there is some scaled copy of \(C^n\), at scaling \(p^{(p-1)/2}\), within \(C^{n'}\) which is \(\sim\)-invariant.
\end{lemma}
\begin{proof}
    Recall that to show \(C^n\) is \(\sim\)-invariant, it suffices to show that \(\sim\) is \([p]\)-interchangeable and \([p]\)-swappable on \(C^n\). Combining Lemmas \ref{lem:canswap} and \ref{lem:canchange}, we note that for any \(r\in\{2,\dots,p\}\), there is some \(n'\) such that, given any \([r-1]\)-swappable equivalence relation \(\sim\) on \(C^{n'}\), there is some scaled copy of \(C^n\) at scaling \(\sqrt{p}\) within \(C^{n'}\) on which \(\sim\) is both \([r]\)-interchangeable and \([r]\)-swappable.
    
    Also note that Lemma \ref{lem:canswap} alone implies that there is some \(n'\) such that, given any equivalence relation \(\sim\) at all on \(C^{n'}\), there is some copy of \(C^n\) within \(C^{n'}\) on which \(\sim\) is \(\{1\}\)-swappable. This serves as a base case to induct, using the above observation, on the claim that for any \(r\in [p]\), there is some \(n'\) such that, given any equivalence relation \(\sim\) at all on \(C^{n'}\), there is some scaled copy of \(C^n\) at scaling \(p^{(r-1)/2}\) within \(C^{n'}\) on which \(\sim\) is both \([r]\)-interchangeable and \([r]\)-swappable. But now for \(r=p\), this gives the desired conditions for this scaled copy of \(C^n\) to be \(\sim\)-invariant, as desired.
\end{proof}

\section{Finding a rainbow or monochromatic \(C^k\)}
Throughout this section, we work within a \(\sim\)-invariant \(C^n\), so that \(\sim_w\) depends only on the dimension of \(w\). We therefore write \(\sim_m\) for the relation \(\sim_w\) induced by words \(w\) of dimension \(m\). Indeed, where there is no possible confusion, we suppress the subscript and simply write \(\sim\), since \(m\) can be deduced from the length of the words either side of the symbol. So for example, \(ab\sim ba\) will be shorthand for \(ab\sim_2 ba\). Indeed, we will be primarily concerned with expressions of the form \(ab\sim ba\), and now prove two lemmas about them.

\begin{lemma}\label{lem:eqcomm}
    Let \(C^n\) be \(\sim\)-invariant, and, for some \(n'<n\), suppose that \(a,a'\in[p]^{n'}\) are such that \(a\sim a'\). Then for each \(i\in[n']\), we have \(a_ia'_i\sim a'_ia_i\).
\end{lemma}
\begin{proof}
    Let \(w\) be the word of length \(n'+1\) that has \(w_i=a_i\), and all other entries equal to \(*\). Similarly, let \(w'\) be the word of length \(n'+1\) that has \(w_{i+1}=a_i\), and all other entries equal to \(*\). We now define \(b,b',b''\in[p]^{n'+1}\) as follows:
    \begin{align*}
        b&=a_1 a_2\dots a_{i-1} a_i a_i a_{i+1} a_{i+2} \dots a_{n'},\\
        b'&=a'_1 a'_2\dots a'_{i-1} a_i a'_i a'_{i+1} a'_{i+2} \dots a'_{n'},\text{ and}\\
        b''&=a'_1 a'_2\dots a'_{i-1} a'_i a_i a'_{i+1} a'_{i+2} \dots a'_{n'}.
    \end{align*}
    In particular, notice that \(b=\iota_w(a)\) and \(b'=\iota_w(a')\), so as \(a\sim a'\), we must have \(b\sim b'\). Likewise, as \(b=\iota_{w'}(a)\) and \(b''=\iota_{w'}(a')\), we must have \(b\sim b''\). Hence \(b'\sim b''\), by transitivity. But now let \(w''\) be the word of length \(n'+1\) defined as
    \[w''=a'_1 a'_2\dots a'_{i-1} * *\ a'_{i+1} a'_{i+2} \dots a'_{n'},\]
    and notice that \(b'=\iota_{w''}(a_i a'_i)\), and \(b''=\iota_{w''} (a'_i a_i)\). Hence as \(b'\sim b''\), we must have \(a_ia'_i\sim a'_ia_i\), as desired.
\end{proof}

\begin{lemma}\label{lem:transcomm}
    Let \(C^n\) be \(\sim\)-invariant, with \(n\geq 4\). Suppose that for some \(a,b,c\in [p]\), we have \(ab\sim ba\) and \(bc \sim cb\). Then \(ac \sim ca\).
\end{lemma}
\begin{proof}
    Now \(abc=\iota_{**c}(ab)\) and \(bac=\iota_{**c}(ba)\), so certainly \(abc\sim bac\). Likewise, since \(abc=\iota_{a**}(bc)\) and \(acb=\iota_{a**}(cb)\), we have \(abc\sim acb\). But now transitivity gives \(bac \sim acb\)---so since \(bac\) has second entry \(a\), and \(acb\) has second entry \(c\), Lemma \ref{lem:eqcomm} gives the desired result.
\end{proof}

We are now in a position to prove Theorem \ref{thm:mainplus}.
\begin{proof}[Proof of Theorem \ref{thm:mainplus}]
    By Lemma \ref{lem:invariant}, we may choose \(n\) such that given any equivalence relation \(\sim\) on \(p^{-p/2}C^n\), there is a \(\sim\)-invariant copy of \(p^{-1/2}C^{pk+1}\) within \(p^{-p/2}C^n\). Therefore, given any colouring of \(p^{-p/2}C^n\), we may consider the equivalence relation \(\sim\) that corresponds to that colouring and pass immediately to a \(\sim\)-invariant copy of \(p^{-1/2}C^{pk+1}\).

    We consider the standard emulated copy of \(C^k\) within \(p^{-1/2}C^{pk+1}\), embedded by the map \(\sigma_{k,pk+1}\). Recall that this maps \((a_1,\dots,a_k)\) to
\[\left(a_1,a_1+1,\dots,a_1+(p-1),a_2,\dots,a_k,a_k+1,\dots,a_k+(p-1),1\right).\]
    Suppose that our colouring of our initial set contained no rainbow copy of \(C^k\). Then certainly there are two distinct points \(a,a'\) in this copy of \(C^k\) which receive the same colour---so \(\sigma_{k,pk+1}(a)\sim \sigma_{k,pk+1}(a')\). Recall that \(\sigma_k(a)\) is just \(\sigma_{k,pk+1}(a)\) with the last coordinate deleted---so in fact we have \(\sigma_{k}(a)\sim_{pk} \sigma_{k}(a')\), and these have the right length to apply Lemma \ref{lem:eqcomm}.
    
    Since \(a\neq a'\), there is some \(i\) such that \(a_i\neq a'_i\). Then for each \(j\in [p]\), we have \(\sigma(a)_{(i-1)p + j}=a_i + j-1\), and \(\sigma(a')_{(i-1)p + j}=a'_i + j-1\). So by Lemma \ref{lem:eqcomm}, we have
    \[(a_i+j-1)(a'_i+j-1)\sim (a'_i+j-1)(a_i+j-1)\]
    within our copy of \(p^{-1/2}C^{pk+1}\).

    But \(j\) was arbitrary, so in fact writing \(d=a'_i-a_i\neq 0\), we have
    \[i(i+d)\sim (i+d)i\]
    for any \(i\), where again we work modulo \(p\). But now recall that by Lemma \ref{lem:transcomm}, whenever we have \(i(i+(\ell-1)d) \sim (i+(\ell-1)d)i\) and \((i+(\ell-1)d)(i+\ell d)\sim (i+\ell d)(i+(\ell-1)d)\), we may also deduce that \(i(i+\ell d) \sim (i+\ell d)i\). Thus inductively, we have \(i(i+\ell d)\sim (i+\ell d)i\) for all \(\ell\). But recall that \(p\) is prime and \(d\) is nonzero modulo \(p\). Thus we may write any element \(j\) of \(\{1,\dots,p-1\}\) in the form \(\ell d\) (modulo \(p\)). Thus \(i(i+j)\sim (i+j)i\)---so in fact, for any \(a,b\), we have \(ab\sim ba\).

    But this implies that applying any adjacent transposition to the coordinates of a point in \([p]^{pk+1}\) does not change its colour---so as these generate \(S_{pk+1}\), any permutation of the coordinates of a point in \([p]^{pk+1}\) does not change its colour. But all the points of the standard emulated copy of \(C^k\) given by \(\sigma\) have exactly \(k+1\) entries equal to \(1\), and \(k\) entries equal to each element of \(\{2,\dots,p\}\). Thus these are certainly all permutations of each other---so they all receive the same colour. Thus this is a monochromatic copy of \(C^k\), as desired.
\end{proof}

\section{Concluding remarks}
One might have hoped, by analogy with traditional Euclidean Ramsey theory, that a proof that the regular pentagon is canonically Ramsey would give rise to a proof that every subsoluble set is canonically Ramsey, which would extend the class of known canonically Ramsey sets to cover almost every known Ramsey set \cite{behague2025nearlyknowneuclideanramsey}. It is therefore perhaps surprising that this proof method not only fails to prove an analogous result, but even fails to show that every regular polygon is canonically Ramsey.

For composite \(r\), it is only at the last step that this proof fails for the regular \(r\)-gon, so one might ask whether it can be fixed by a different argument. In fact, this obstruction seems to be more fundamental: for composite \(r\geq 6\), where \(C\) is a regular \(r\)-gon, it is possible to construct a colouring of any \(C^n\) which contains no monochromatic or rainbow scaled copies of \(C\). For example, in the case of the hexagon, we may colour \(C^n\) with the following rule: two points of \(C^n\) have the same colour if and only if, in each coordinate, their respective entries are either the same point of \(C\) or opposite points of \(C\). Then any scaled copy of \(C\) will receive exactly three colours: each point will receive the same colour as the point opposite it, and no other. There are similar constructions for all composite \(r\).

In a sense, this colouring of \(C^n\) is `canonical' in the following sense: every scaled copy of some \(C^k\) contained within \(C^n\) receives a colouring of the same form. We note that our proof can be modified to show that within sufficiently large \(C^n\), one can find scaled copies of \(C^k\) which receive a colouring of this form---that is, two points receive the same colour if and only if, identifying \(C^k\) with \([r]^k\), the difference of their respective coordinates in each entry is a multiple of some fixed \(d|r\).

Thus, for any composite \(r\), any configuration witnessing that the regular \(r\)-gon is canonically Ramsey cannot be of the form \(C^n\). In light of this discussion, we ask the following question:
\begin{question}
    Is the regular hexagon canonically Ramsey?
\end{question}

\section{Acknowledgement}
The author is funded by an Internal Graduate Studentship of Trinity College, Cambridge.

	\bibliographystyle{plain}
\bibliography{main}

\end{document}